\documentclass[11pt]{article}
\usepackage{amsmath,amsfonts,amssymb,latexsym,amsthm,dsfont}
\usepackage{geometry,graphicx}
\usepackage{hyperref}

\title{Trend to equilibrium and particle approximation for a weakly
 selfconsistent 
 Vlasov-Fokker-Planck equation}

\author{%
 Fran\c cois~\textsc{Bolley}, %
 Arnaud~\textsc{Guillin}, %
 Florent~\textsc{Malrieu} } 

\date{Preprint -- \today}

\newtheorem{thm}{Theorem}
\newtheorem{prop}[thm]{Proposition}%
\newtheorem{lem}[thm]{Lemma}%
\newtheorem{lemma}[thm]{Lemma}%

\newtheorem{rem}[thm]{Remark}%
\newcommand{\dE}{\mathbb{E}}

\newcommand{\dP}{\mathbb{P}}
\newcommand{\dQ}{\mathbb{Q}}\newcommand{\dR}{\mathbb{R}}

\newcommand{\rr}{\mathbb{R}}
\newcommand{\ee}{\mathbb{E}}

\newcommand{\cA}{\mathcal{A}}\newcommand{\cB}{\mathcal{B}}

\newcommand{\ABS}[1]{{{\left| #1 \right|}}} 
\newcommand{\PAR}[1]{{{\left(#1\right)}}} 
\renewcommand{\leq}{\leqslant}
\renewcommand{\geq}{\geqslant}
\newcommand{\ds}{\displaystyle}

\begin{document}

\maketitle

\begin{abstract}
 We consider a Vlasov-Fokker-Planck equation governing the evolution
 of the density of interacting and diffusive matter in the space of
 positions and velocities.
 We use a probabilistic interpretation to obtain convergence towards
 equilibrium in Wasserstein distance with an explicit exponential
 rate. We also prove a propagation of chaos property for an
 associated particle system, and give rates on the approximation of
 the solution by the particle system.  Finally, a transportation
 inequality for the distribution of the particle system leads to
 quantitative deviation bounds on the approximation of the
 equilibrium solution of the equation by an empirical mean of the
 particles at given time.

\end{abstract}

\section*{Introduction and main results}

We are interested in the long time behaviour and in a particle
approximation of a distribution $f_t(x,v)$ in the space of positions
$x \in \rr^d$ and velocities $v \in \rr^d$ (with $d\geq 1$) evolving according to the
Vlasov-Fokker-Planck equation
\begin{equation}
\label{eq:vfp}
\displaystyle{ \frac{\partial f_t}{\partial t}}+v\cdot \nabla_x f_t - C*_x\rho[f_t] (x)  \cdot \nabla_v f_t%
=\Delta_v f_t+ \nabla_v\cdot( (A(v) + B(x)) f_t), \qquad t >0, \, x,v \in \rr^d
\end{equation}
where
$$
\rho[f_t](x)=\int_{\rr^d} f_t(x,v) \, dv
$$
is the macroscopic density in the space of positions $x \in \rr^d$ (or
the space marginal of $f_t$). Here $a \cdot b$ denotes the scalar
product of two vectors $a$ and $b$ in $\rr^{d}$ and $*_x$ stands for
the convolution with respect to $x \in \rr^{d}:$
$$
C*_x \rho[f_t](x) = \int_{\rr^d} C(x-y) \, \rho[f_t](y) \, dy =
\int_{\rr^{2d}} C(x-y) \, f_t(y,v) \, dy \, dv.
$$
Moreover $\nabla_x$ stands for the gradient with respect to the
position variable $x \in \rr^d$ whereas $\nabla_v$, $\nabla_v \cdot$
and $\Delta_v$ respectively stand for the gradient, divergence and
Laplace operators with respect to the velocity variable $v \in
\rr^{d}$.

\medskip

The $A(v)$ term models the friction, the $B(x)$ term models an
exterior confinement and the $C(x-y)$ term in the convolution models
the interaction between positions $x$ and $y$ in the underlying
physical system. For that reason we assume that $C$ is an odd map on
$\rr^d.$ This equation is used in the modelling of the distribution
$f_t(x,v)$ of diffusive, confined and interacting stellar or charged
matter when $C$ respectively derives from the Newton and Coulomb
potential (see~\cite{bouchut-dolbeault95} for instance). It has the
following natural probabilistic interpretation: if $f_0$ is a density
function, the solution $f_t$ of \eqref{eq:vfp} is the density of the 
law at time $t$ of the $\rr^{2d}$-valued process $(X_t, V_t)_{t\geq 0}$ 
evolving according to the mean field stochastic differential equation
(diffusive Newton's equations)
\begin{equation}\label{edsVFP}
 \left\{ \begin{array}{rcl}
     dX_t    & = &  V_t \, dt \\
     dV_t     & = & - A(V_t) \, dt- B(X_t) \, dt - C *_x \rho[f_t](X_t) \, dt+ \sqrt{2} \, dW_t.   
   \end{array}
 \right. 
\end{equation}
Here $(W_t)_{t \geq 0}$ is a Brownian motion in the velocity space
$\rr^{d}$ and $f_t$ is the law of $(X_t,V_t)$ in $\rr^{2d}$, so that $\rho[f_t]$ is the law of $X_t$ in $\rr^d.$

\bigskip

Space homogeneous models of diffusive and interacting granular media
(see \cite{BCCP}) have been studied by P.~Cattiaux and the last two authors in
particular~\cite{CGM}, \cite{malrieu01, malrieu03}, by means of a
stochastic interpretation analogous to~\eqref{edsVFP} and a particle
approximation analogous to~\eqref{eq:sdesys} below. They were
interpreted as gradient flows in the space of probability measures by
J. A. Carrillo, R. J. McCann and C. Villani \cite{CMV,CMV2} (see
also~\cite{villani-stflour}), both approaches leading to explicit
exponential (or algebraic for non uniformly convex potentials) {\it rates
of convergence to equilibrium}. Also possibly time-uniform propagation of chaos was proven for the associated particle system.

\medskip

Obtaining rates of convergence to equilibrium for~\eqref{eq:vfp} is
much more complex, as the equation simultaneously presents hamiltonian
and gradient flows aspects. Much attention has recently been called
to the linear noninteracting case of~\eqref{eq:vfp}, when $C=0$, also
known as the kinetic Fokker-Planck equation. First of all a
probabilistic approach based on Lyapunov functionals, and thus easy
to check conditions, lead D. Talay \cite{talay}, L. Wu \cite{wudamp}
or D. Bakry, P. Cattiaux and the second author \cite{BCG} to
exponential or subexponential convergence to equilibrium in total
variation distance. The case when $A(v)=v$ and $B(x) = \nabla \Psi(x)$, and when the equilibrium solution is explicitely given by $f_{\infty}(x,v) = e^{-\Psi(x) - \vert v
  \vert^2/2}$ is studied in \cite{herau07}, \cite{herau-nier04} and \cite[Chapter 7]{villani-hypo}: hypocoercivity analytic techniques are developed which, applied to this situation, give sufficient
conditions, in terms of Poincar\'e or logarithmic Sobolev
inequalities for the measure $e^{-\Psi}$, to $L^2$ or entropic
convergence with an explicit exponential rate. We also refer to~\cite{EGM} for the evolution of two species, modelled by two coupled Vlasov-Fokker-Planck equations. 

\medskip

C. Villani's approach extends to the selfconsistent situation when $C$
derives from a nonzero potential $U$ (see~\cite[Chapter
17]{villani-hypo}): replacing the confinement force $B(x)$ by a
periodic boundary condition, and for small and smooth potential $U$,
he obtains an explicit exponential rate of convergence of all
solutions toward the unique normalized equilibrium solution $e^{-\vert
 v \vert^2/2}.$

In this work we consider the case when the equation is set on the
whole $\rr^d$, with quadratic-like friction $A(v)$ and confinement
$B(x)$ forces, and small Lipschitz interaction $C(x)$: in the whole
paper we make the following

\noindent
{\bf Assumption.}
 We say that Assumption $(\cA)$ is fullfilled if there exist
 nonnegative constants $\alpha, \alpha', \beta, \gamma$ and $\delta$
 such that
$$
\vert A(v) - A(w) \vert \leq \alpha \vert v-w \vert, \quad (v-w) \cdot
(A(v)-A(w)) \geq \alpha' \vert v - w \vert^2,
$$
$$
B(x) = \beta \, x + D(x) \quad \textrm{where} \; \vert D(x) - D(y)
\vert \leq \delta \vert x-y \vert
$$
and
$$
\vert C(x) - C(y) \vert \leq \gamma \vert x-y \vert
$$
for all $x,y,v,w$ in $\rr^d.$ 

Convergence of solutions will be measured in terms of Wasserstein distances: let ${\mathcal P}_2$ be the space of Borel probability measures $\mu$ on $\rr^{2d}$ with finite second moment, that is, such that the integral $\ds \int_{\rr^{2d}} (\vert x \vert^2 + \vert v \vert^2) \, d\mu(x,v)$ be
finite. The space ${\mathcal P}_2$ is equipped with the
(Monge-Kantorovich) Wasserstein distance $d$ of order $2$ defined by
$$
d(\mu, \nu)^2 = \inf_{(X,V), \, (Y,W)} %
\dE\PAR{\ABS{ X-Y }^2 + \ABS{V - W}^2}
$$
where the infimum runs over all the couples $(X,V)$ and $(Y,W)$ of
random variables on $\rr^{2d}$ with respective laws $\mu$ and $\nu.$ Convergence in this metric is equivalent to narrow convergence plus convergence of the second moment (see~\cite[Chapter 6]{villani-stflour} for instance).

The coefficients $A, B$ and $C$ being Lipschitz, existence and uniqueness for Equation~\eqref{edsVFP} with square-integrable initial data are ensured by \cite{meleard}. It follows that, for all initial data $f_0$ in ${\mathcal P}_2$, Equation~\eqref{eq:vfp} admits a unique measure solution in  ${\mathcal P}_2$, that is, continuous on $[0, + \infty[$ with values in  ${\mathcal P}_2$.

Assumption $(\cA)$ is made in the whole paper. Under an additional
assumption on the smallness of $\gamma$ and $\delta$, we shall prove
a quantitative exponential convergence of all solutions to a unique
equilibrium:

\begin{thm}\label{theocveq}
 Under Assumption $(\cA)$, for all positive $\alpha, \alpha'$ and
 $\beta$ there exists a positive constant $c$ such that, if $0 \leq
 \gamma, \delta < c$, then there exist positive constants $C$ and
 $C'$ such that
\begin{equation}\label{pseudocontr}
d(f_t, \bar{f}_t) \leq C' \, e^{-C t} \, d(f_0,\bar{f}_0), \qquad t \geq 0
\end{equation}
for all solutions $(f_t)_{t \geq 0}$ and $(\bar{f}_t)_{t \geq 0}$
to~\eqref{eq:vfp} with respective initial data $f_0$ and $\bar{f}_0$ in
${\mathcal P}_2.$

Moreover~\eqref{eq:vfp} admits a unique stationnary solution
$\mu_{\infty}$ and all solutions $(f_t)_{t \geq 0}$ converge towards
it, with
$$
d(f_t, \mu_{\infty}) \leq C' e^{-C t} d(f_0,\mu_{\infty}), \qquad  t \geq 0.
$$
\end{thm}

For instance, for $\alpha = \alpha' = \beta =1$, the general proof
below shows that the nonnegative $\gamma$ and $\delta$ with $\gamma +
\delta < 0,26$ are admissible. In the linear free case when $\gamma =
\delta =0,$ the convergence rate is given by $C = 1/3$, and for
instance for $\gamma$ and $\delta$ with $\gamma + \delta = 0,1$ we
obtain $C \sim 0,27.$

\medskip

Compared to Villani's results, convergence is here proven in the (weak) Wasserstein distance, not in $L^1$ norm, or relative entropy as in the noninteracting case - the latter being a stronger convergence since, in this specific situation, the equilibrium measure $e^{- \Psi(x) - \vert v \vert^2/2}$ satisfies a logarithmic Sobolev inequality, hence a transportation inequality. We refer to~\cite{logsob}, \cite{ledoux01} or \cite[Chapter 22]{villani-stflour} for this and forthcoming notions. 

However our result holds in the noncompact case with small Lipschitz interaction, and can be seen as a first attempt to deal with more general case. Moreover it shows existence and uniqueness of the equilibrium measure, and in particular does not use its explicit expression (which is unknown in our broader situation). It is also not only a result on the convergence to equilibrium, but also a stability result of all solutions. Let us finally note that it is based on the natural stochastic interpretation~\eqref{edsVFP} and a simple coupling argument, and does not need any hypoelliptic regularity property of the solutions.

\bigskip

The {\it particle approximation} of solutions to~\eqref{eq:vfp} consists in the introduction of a large number $N$ of $\rr^{2d}$-valued processes $(X^{i,N}_t,
V^{i,N}_t)_{t\geq 0}$ with $1 \leq i \leq N$, no more evolving  according to the force field $C \ast_x \nobreak \rho[f_t]$ generated by the distribution $f_t$ as in~\eqref{edsVFP}, but by the empirical measure $\ds \hat{\mu}^N_t = \frac{1}{N} \sum_{i=1}^N \delta_{(X^{i,N}_t, V^{i,N}_t)}$ of the system: if ${(W_\cdot^i)}_{i\geq 1}$ with $i \geq 1$ are independent standard Brownian motions
on $\rr^d$ and $(X^i_0,V^i_0)$ with $i \geq 1$ are independent random vectors on $\rr^{2d}$
with law $f_0$ in ${\mathcal P}_2$ and indepedent of ${(W^i_\cdot)}_{i\geq 1}$, we let  ${(X^{(N)}_t,V^{(N)}_t)}_{t\geq  0}={(X^{1,N}_t,\ldots,X^{N,N}_t,V^{1,N}_t,\ldots,V^{N,N}_t)}_{t\geq 0}$ be the solution of the following stochastic differential equation in $(\rr^{2d})^N$:
\begin{equation}
\label{eq:sdesys}
\begin{cases}
dX^{i,N}_t=V^{i,N}_t\,dt \\
\displaystyle{ dV^{i,N}_t=%
 - A(V^{i,N}_t) \, dt - B(X^{i,N}_t)  \, dt - \frac{1}{N}\sum_{j=1}^N C(X^{i,N}_t-X^{j,N}_t)\,dt%
  +\sqrt{2} \, dW^i_t,} \qquad 1 \leq i \leq N\\
(X^{i,N}_0,V_0^{i,N})=(X^{i}_0,V_0^{i}).
\end{cases}
\end{equation}


The mean field force $C \ast_x \rho[f_t]$ in~\eqref{edsVFP} is replaced by the pairwise actions $\ds \frac{1}{N} C(X^{i,N}_t-X^{j,N}_t)$ of particle $j$ on particle $i.$ Since this interaction is of order $1/N$, it may be reasonable that two of these interacting particles (or a fixed number $k$ of them) become less and less correlated as $N$ gets large. 

In order to state this {\it propagation of chaos} property we let, for each $i \geq 1$,  $(\bar{X}^i_t,\bar{V}^i_t)_{t \geq 0}$ be the solution of the kinetic McKean-Vlasov type equation on $\dR^{2d}$
\begin{equation}
\label{eq:sdemkv}
\begin{cases}
  d\bar{X}_t^i=\bar{V}_t^i\,dt \\
  d\bar{V}_t^i=  - A(\bar{V}_t^i) \, dt - B(\bar{X}_t^i) \, dt  - C \ast_x \rho[\nu_t] (\bar{X}_t^i)\,dt + \sqrt{2} \, dW_t^i,\\
  (\bar{X}_0^i,\bar{V}_0^i)=(X_0^i,V_0^i).
\end{cases}
\end{equation}
where $\nu_t$ is the distribution of $(\bar{X}^i_t,\bar{V}^i_t).$ The processes ${(\bar{X}^i_t,\bar{V}^i_t)_{t \geq 0}}$ with $i\geq 1$ are independent since the initial conditions and driving Brownian motions are independent. Moreover they are identically distributed and their common law at time $t$ evolves according to~\eqref{eq:vfp}, so is the solution $f_t$ of~\eqref{eq:vfp} with initial datum $f_0.$ In this notation, and as $N$ gets large, the $N$ processes $(X^{i,N}_t, V^{i,N}_t)_{t\geq 0}$ look more and more like the $N$ independent processes ${(\bar{X}^i_t,\bar{V}^i_t)_{t \geq 0}}$:

\begin{thm}[Time-uniform propagation of chaos]\label{th:chaos}
 Let $(X^i_0, V^i_0)$ for $1 \leq i \leq N$ be $N$ independent
 $\rr^{2d}$-valued random variables with law $f_0$ in ${\mathcal
   P}_2(\rr^{2d}).$ Let also $(X^{i,N}_t, V^{i,N}_t)_{t \geq 0, 1
   \leq i \leq N}$ be the solution to~\eqref{edsVFP} and
 $(\bar{X}^i_t, \bar{V}^i_t)_{t \geq 0}$ the
 solution to~\eqref{eq:sdemkv} with initial datum $(X^i_0, V^i_0)$ for $1
   \leq i \leq N.$ Under Assumption $(\cA)$, for all positive
 $\alpha, \alpha'$ and $\beta$ there exists a positive constant $c$
 such that, if $0 \leq \gamma, \delta < c$, then there exists a positive
 constant $C$, independent of $N$, such that for $i=1,\ldots,N$
$$
\sup_{t\geq 0} \,
\dE \, \Big(\ABS{X^{i,N}_t-\bar{X}^i_t}^2+\ABS{V^{i,N}_t-\bar{V}^i_t}^2 \Big)%
\leq \frac{C}{N} \cdot
$$
\end{thm}

Here the constant $C$ depends only on the coefficients of the equation and the second moment of~$f_0.$

\begin{rem}\label{rem:cvmarg}
In particular the law $f_t^{(1,N)}$ at time $t$ of any $(X^{i,N}_t, V^{i,N}_t)$ (by symmetry) converges to $f_t$ as $N$ goes to infinity, according to
$$
d(f_t^{(1,N)}, f_t)^2 \leq \dE \, \Big( \ABS{X^{i,N}_t-\bar{X}^i_t}^2+\ABS{V^{i,N}_t-\bar{V}^i_t}^2 \Big) \leq \frac{C}{N} \cdot
$$
\end{rem}
%


\bigskip

Propagation of chaos at the level of the {\it trajectories}, and not only of the time-marginals, is estimated in~\cite{bolley-paths} and \cite{meleard} for a broad class of equations, but with non time-uniform constants.

\bigskip 

We finally turn to the approximation of the equilibrium solution of the Vlasov-Fokker-Planck equation (as given by Theorem~\ref{theocveq}) by the particle system at a given time $T$.

Since all solutions $(f_t)_{t \geq 0}$ to~\eqref{eq:vfp} with initial data in ${\mathcal P}_2$ converge to the equilibrium solution $\mu_{\infty}$, we let $(x_0, v_0)$ in $\rr^{2d}$ be given and we consider the Dirac mass $\delta_{(x_0, v_0)}$ at $(x_0, v_0)$ as the initial datum $f_0$. We shall give precise bounds on the approximation of $\mu_{\infty}$ by the empirical measure of the particles $(X^{i,N}_t,
V^{i,N}_t)$ for $1 \leq i \leq N$, all of them initially at $(x_0, v_0).$

In the space homogeneous case of the granular media equation, this was performed by the third author~\cite{malrieu01,malrieu03} by proving a logarithmic Sobolev inequality for the joint law $f_t^{(N)}$ of the $N$ particles at time $t$. In turn this inequality was proved by a Bakry-Emery curvature criterion (see~\cite{bakry-emery}). The argument does not work here as the particle system has $- \infty$ curvature, and we shall only prove a (Talagrand) $T_2$ transportation inequality for the joint law of the particles.

\begin{rem} 
At this stage we have to point out that, for instance when the force fields $A, B$ and $C$ are gradient of potentials, the invariant measure of the particle system, that is, the large time limit of the joint law of the $N$ particles, is explicit and satisfies a logarithmic Sobolev inequality with carr\'e du champ $\vert \nabla_x f \vert^2 + \vert \nabla_v f \vert^2;$ however it does not satisfy a logarithmic Sobolev inequality with carr\'e du champ $\vert \nabla_v f \vert^2$ (initiated by our dynamics), which would at once lead to exponential entropic convergence to equilibrium for the particle system. 
\end{rem}



Let us recall that a probability measure $\mu$ on $\rr^{2d}$ is said to satisfies a $T_2$ transportation inequality if there exists a constant $D$ such that 
$$
d(\mu,\nu)^2\le D~H(\nu|\mu)
$$
for all probability measure $\nu$; here 
$$
H(\nu|\mu)=\int \log\left(\frac{d\nu}{d\mu}\right)d\nu
$$
if $\nu\ll\mu$ and $+\infty$ otherwise is the relative entropy of $\nu$ with respect to $\mu$.

\medskip

\begin{thm}\label{th:T2}
Under the assumptions of Theorem \ref{theocveq}, for all positive $\alpha,\alpha'$ and $\beta$ there exists a positive constant $c$ such that if $0 \leq \gamma,\delta<c$, then the joint law of the $N$ particles $(X_T^{(i,N)},V_T^{(i,N)})$ at given time $T$, all with deterministic starting points $(x_0,v_0) \in \rr^{2d}$, satisfies a $T_2$ inequality with a constant $D$ independent of the number $N$ of particles, of time $T$ and of the point $(x_0, v_0)$.\\
It follows that there exists a constant $D'$ such that
$$
\dP \left(\frac1N\sum_{i=1}^Nh(X^{i,N}_T,V^{i,N}_T)- \int_{\rr^{2d}}  h\, d\mu_\infty(h)\ge r + D'\left(\frac{1}{\sqrt{N}}+e^{-CT}\right)\right)\le\exp\left(-\frac{Nr^2}{2D}\right)
$$
for all $N, T, r \geq 0$ and all $1$-Lipschitz observables $h$ on $\rr^{2d}.$
\end{thm}

Here the constant $C$ has been obtained in Theorem~\ref{theocveq} and the constant $D'$ depends only on the point $(x_0, v_0)$ and the coefficients of the equation.


\begin{rem}
Such single observable deviation inequalities were obtained in~\cite{malrieu01} for the space homogeneous granular media equation; they were upgraded in  \cite{BGV} to the very level of the measures, and to the level of the density of the equilibrium solution. The authors believe that such estimates can also be obtained in the present case.
\end{rem}

\begin{rem}
Let us also point out that if we do not suppose a confinement/convexity assumption as in (${\cal A}$) but only Lipschitz regularity on the drift fields $A, B$ and $C$, then Theorems~\ref{theocveq}, \ref{th:chaos} and~\ref{th:T2} still hold but with constants growing exponentially fast with time $T$.
\end{rem}

Sections~\ref{sect-cveq},~\ref{sect-prop} and \ref{sect-TI} are respectively devoted to the proofs of Theorems~\ref{theocveq},~\ref{th:chaos} and \ref{th:T2}.


\section{Long time behaviour for the Vlasov-Fokker-Planck
 equation}\label{sect-cveq}


This section is devoted to the proof of Theorem~\ref{theocveq}, which is based on the stochastic interpretation~\eqref{edsVFP} of~\eqref{eq:vfp} and a coupling argument. It uses the idea, also present in~\cite{talay} and~\cite{villani-hypo}, of perturbing the
Euclidean metric on $\dR^{2d}$ in such a way that \eqref{edsVFP} is
dissipative for this metric.

If $Q$ is a positive quadratic form on $\rr^{2d}$ and $\mu$ and $\nu$
are two probability measures in $\mathcal P_2$ we let
$$ 
d_Q(\mu, \nu)^2 = \inf_{(X,V), \, (Y,W)} \dE\PAR{Q((X,V) - (Y,W))}
$$
where again the infimum runs over all the couples $(X,V)$ and $(Y,W)$
of random variables on $\rr^{2d}$ with respective laws $\mu$ and
$\nu$; so that $d_Q = d$ if $Q$ is the squared Euclidean norm on
$\rr^{2d}.$ The key step in the proof is the following

\begin{prop}\label{propcveq}
 Under the assumptions of Theorem~\ref{theocveq}, there exist a
 positive constant $C$ and a positive quadratic form $Q$ on
 $\rr^{2d}$ such that
 $$
 d_Q(f_t, \bar{f}_t) \leq e^{-C t} \, d_Q(f_0,\bar{f}_0), \qquad t
 \geq 0
 $$
 for all solutions $(f_t)_{t \geq 0}$ and $(\bar{f}_t)_{t \geq 0}$
 to~\eqref{eq:vfp} with respective initial data $f_0$ and $\bar{f}_0$
 in ${\mathcal P}_2$.
\end{prop}

\begin{proof}[Proof of Proposition~\ref{propcveq}]
 Let $(f_t)_{t \geq 0}$ and $(\bar{f}_t)_{t \geq 0}$ be two solutions
 to~\eqref{eq:vfp} with initial data $f_0$ and $\bar{f}_0$ in
 $\mathcal P_2.$ Let also $(X_0, V_0)$ and $(\bar{X}_0, \bar{V}_0)$
 with respectively law $f_0$ and $\bar{f}_0$, evolving into $(X_t,
 V_t)$ and $(\bar{X}_t, \bar{V}_t)$ according to~\eqref{edsVFP}, both
 with the same Brownian motion $(W_t)_{t \geq 0}$ in $\rr^d.$ Then,
 by difference, $(x_t, v_t) = (X_t - \bar{X}_t, V_t - \bar{V}_t) $
 evolves according to
$$
\left\{ \begin{array}{rcl}
   dx_t    & = &  v_t \, dt \\
   dv_t & = & - \big( A(V_t) - A(\bar{V}_t)+ \beta \, x_t + D(X_t) -
   D(\bar{X}_t) \big) \, dt - \big( C *_x \rho[f_t](X_t) - C *_x
   \rho[\bar{f}_t](\bar{X}_t) \big) \, dt.
                         \end{array}
                       \right.
$$

Then, if $a$ and $b$ are positive constants to be chosen later on, 
\begin{eqnarray*}
 \frac{d}{dt} ( a \vert x_t \vert^2 + 2 \, x_t \cdot v_t + b \vert v_t \vert^2 ) 
 & = &
 2 \, a  \, x_t \cdot  v_t + 2 \vert v_t \vert^2  \\
 & &- 2 \, x_t \cdot \big( A(V_t) - A(\bar{V}_t) + \beta x_t + D(X_t) - D(\bar{X}_t)  \big) \\
 & & - 2 \, b \, v_t \cdot \big( A(V_t) - A(\bar{V}_t)  + \beta x_t  + D(X_t) - D(\bar{X}_t) \big)\\
 & & -  2 \, (x_t + b \, v_t) \cdot \big( C *_x \rho[f_t](X_t) - C *_x \rho[\bar{f}_t](\bar{X}_t) \big).
\end{eqnarray*}

By the Cauchy-Schwarz inequality and assumptions on $A$ and $D$, the
first four terms are bounded by above by
$$
2 (a - b \beta) x_t \cdot v_t + 2 (\alpha + b \delta) \vert x_t \vert \, \vert v_t \vert - 2 \, (b \alpha' -1) \vert v_t \vert^2 - 2 \, (\beta - \delta) \, \vert x_t \vert^2.
$$

\medskip

Let now $\pi_t$ be the law of $(X_t, V_t; \bar{X}_t, \bar{V_t})$ on
$\rr^{2d} \times \rr^{2d}:$ then its marginals on $\rr^{2d}$ are the
respective distributions $f_t$ and $\bar{f}_t$ of $(X_t, V_t)$ and
$(\bar{X}_t, \bar{V}_t)$, so that, since moreover $C$ is odd:
\begin{eqnarray*}
& &- 2 \, \ee \, x_t \cdot \big( C *_x \rho[f_t](X_t) - C *_x \rho[\bar{f}_t](\bar{X}_t) \big)\\
& = & - 2  \int_{\rr^{8d}} (Y- \bar{Y}) \cdot \big(C(Y-y) - C(\bar{Y} - \bar{y}) \big) \, d\pi_t(y,w; \bar{y}, \bar{w}) \, d\pi_t(Y,W; \bar{Y}, \bar{W})\\
& = & - \int_{\rr^{8d}} \big( (Y-y) - (\bar{Y} - \bar{y}) \big) \cdot \big( C(Y-y) - C(\bar{Y} - \bar{y}) \big) \, d\pi_t(y,w; \bar{y}, \bar{w}) \, d\pi_t(Y,W; \bar{Y}, \bar{W})\\
& \leq & \gamma  \int_{\rr^{8d}} \big\vert (Y-y) - (\bar{Y} - \bar{y}) \big\vert^2 \, d\pi_t(y,w; \bar{y}, \bar{w}) \, d\pi_t(Y,W; \bar{Y}, \bar{W})\\
& = & 2  \, \gamma \Big[ \int_{\rr^{4d}} \vert y - \bar{y} \vert^2 d\pi_t(y,w; \bar{y}, \bar{w}) - \Big\vert \int_{\rr^{4d}} ( y - \bar{y} ) d\pi_t(y,w; \bar{y}, \bar{w}) \Big\vert^2 \Big]\\
& \leq &  2 \, \gamma \, \ee \vert x_t \vert^2.
\end{eqnarray*}

In the same way, and by Young's inequality,
\begin{eqnarray*}
&& -2 \,  \ee v_t \cdot \big( C *_x \rho[f_t](X_t) - C *_x \rho[\bar{f}_t](\bar{X}_t) \big)\\
& = & 
- 2 \int_{\rr^{8d}} (W- \bar{W}) \cdot \big(C(Y-y) - C(\bar{Y} - \bar{y}) \big) \, d\pi_t(y,w; \bar{y}, \bar{w}) \, d\pi_t(Y,W; \bar{Y}, \bar{W})\\
& = & 
-  \int_{\rr^{8d}} \big( (W-\bar{W}) - (w - \bar{w}) \big) \cdot \big( C(Y-y) - C(\bar{Y} - \bar{y}) \big) \, d\pi_t(y,w; \bar{y}, \bar{w}) \, d\pi_t(Y,W; \bar{Y}, \bar{W})\\
& \leq & 
\frac{\gamma}{2}  \int_{\rr^{8d}}   \big\vert (W - \bar{W}) - (w - \bar{w}) \big\vert^2 +  \big\vert (Y-\bar{Y} - (y-\bar{y}) \big\vert^2\, d\pi_t(y,w; \bar{y}, \bar{w}) \, d\pi_t(Y,W; \bar{Y}, \bar{W})\\
& \leq & 
\gamma \, \ee [\vert x_t \vert^2 + \vert v_t \vert^2].
\end{eqnarray*}

Collecting all terms leads to the bound
\begin{eqnarray*}
\frac{d}{dt} \ee \, \big( a \vert x_t \vert^2 + 2 \, x_t \cdot v_t + b \vert v_t \vert^2 )
& \leq &
2 (a-b \beta) \, \ee \, x_t \cdot v_t + 2 (\alpha + b \delta) \, \ee \vert x_t \vert \, \vert v_t \vert \\
& & - 2 \Big(\beta - \delta - \gamma - \frac{\gamma \, b}{2} \Big) \ee \vert x_t \vert^2 - 2 \Big( \alpha' b - 1 - \frac{\gamma \, b}{2} \Big) \ee \vert v_t \vert^2 
\end{eqnarray*}
for every positive $\varepsilon$, and then (with $a = b \beta$) to
$$
\frac{d}{dt} \ee \, \big( b \beta \, \vert x_t \vert^2 + 2 \, x_t \cdot v_t + b \vert v_t \vert^2 ) 
\leq -  \Big(2 \beta - 2 \eta - \varepsilon - \eta b \Big) \ee \vert x_t \vert^2 -  \Big((2 \alpha' -\eta) b - 2 - \frac{\alpha^2}{\varepsilon}\Big) \ee \vert v_t \vert^2 
$$
by Young's inequality, where $\eta = \gamma + \delta$.

If $4 - 4 \beta b^2 <0$, that is, if $b > 1/\sqrt{\beta}$, then $Q:
(x,v) \mapsto b \beta \vert x \vert^2 + 2 x \cdot v + b \vert v
\vert^2$ is a positive quadratic form on $\rr^{2d}.$ Then we look for
$b$ and $\varepsilon$ such that
\begin{equation}
 \label{eq:cond1}
2 \beta - 2 \eta - \varepsilon - \eta b  >0 \quad \textrm{and}  \quad  (2 \alpha' - \eta) b - 2 - \frac{\alpha^2}{\varepsilon}  >0, 
\end{equation}
in such a way that
$$
\frac{d}{dt} \ee \, Q(x_t, v_t) \leq - C \, \ee \big[ \vert x_t
\vert^2 + \vert v_t \vert^2 \big]
$$
holds for a positive constant $C.$

Necessarily $\eta < 2 \alpha'$, which is assumed in the sequel. Then,
for instance for $\varepsilon = \beta,$ the conditions
\eqref{eq:cond1} are equivalent to
$$
\frac{2+ \alpha^2/\beta}{2 \alpha' - \eta} < b < \frac{\beta - 2
 \eta}{\eta}, \quad \eta < 2 \alpha'.
$$

We look for $\eta$ such that $\ds \frac{2+ \alpha^2/\beta}{2 \alpha' -
 \eta} < \frac{\beta - 2 \eta}{\eta},$ that is, $\ds 2 \eta^2 - \eta
( 2 + \frac{\alpha^2}{\beta} + \beta + 4 \alpha') + 2 \alpha' \beta
0.$ This polynomial takes negative values at $\eta = 2 \alpha'$, so
it is positive on an interval $[0, \eta_0[$ for some $\eta_0 < 2
\alpha'.$ We further notice that $\ds \eta_0 < \frac{\beta
 \sqrt{\beta}}{1 + 2 \sqrt{\beta}},$ so that there exists $b$ with
all the above conditions for any $0 \leq \eta < \eta_0.$

Hence there exists a constant $\eta_0$, depending only on $\alpha,
\alpha'$ and $\beta$, such that, if $\gamma + \delta < \eta_0$, then
there exist a positive quadratic form $Q$ on $\rr^{2d}$ and a constant
$C$, depending only on $\alpha, \alpha', \beta, \gamma$ and $\delta$
such that
$$
\frac{d}{dt} \ee \, Q(x_t, v_t) \leq - C \, \ee \big[ \vert x_t \vert^2 + \vert v_t \vert^2 \big] 
$$
for all $t\geq 0$. In turn, since $Q(x,v)$ and $\vert x \vert^2 +
\vert v \vert^2$ are equivalent on $\rr^{2d}$, this is bounded by $ -
C \, \ee \, Q(x_t, v_t)$ for a new constant $C$, so that
$$
\ee \, Q \big((X_t, V_t) - (\bar{X}_t, \bar{V}_t) \big) \leq e^{-Ct}
\, \ee \, Q \big( (X_0, V_0) - (\bar{X}_0, \bar{V}_0) \big)
$$
for all $t\geq 0$ by integration. We finally optimize over $(X_0, V_0)$ and
$(\bar{X}_0, \bar{V}_0)$ with respective laws $f_0$ and $\bar{f}_0$
and use the relation $d_Q(f_t, \bar{f}_t) \leq \ee \, Q((X_t, V_t) -
(\bar{X}_t, \bar{V}_t))$ to deduce
$$
d_Q(f_t, \bar{f}_t) \leq e^{-Ct} \, d_Q(f_0, \bar{f}_0).
$$
This concludes the argument. 
\end{proof}

\begin{rem}
This coupling argument can also be performed for the (space homogeneous) granular media equation, for which it exactly recovers the contraction property in Wasserstein distance given in \cite[Theorem 5]{CMV2}, whence the statements which follow on the trend to equilibrium. 
\end{rem}

We now turn to the 

\begin{proof}[Proof of Theorem \ref{theocveq}]
First of all, the positive quadratic form $Q(x,v)$ on $\rr^{2d}$ given
by Proposition~\ref{propcveq} is equivalent to $\vert x \vert^2 +
\vert v \vert^2$, so there exist positive constants $C''$ and $C'$
such that
$$
d(f_t, \bar{f}_t) \leq C'' d_Q(f_t, \bar{f}_t) \leq C'' e^{-Ct} d_Q(f_0, \bar{f}_0) \leq  C' \, e^{-C t} \, d(f_0,\bar{f}_0), \qquad t \geq 0
$$
for all solutions $(f_t)_{t \geq 0}$ and $(\bar{f}_t)_{t \geq 0}$
to~\eqref{eq:vfp} by the contraction property of
Proposition~\ref{propcveq}: this proves the first
assertion~\eqref{pseudocontr} of Theorem~\ref{theocveq}.

\medskip

Now, if $Q$ is the positive quadratic form on $\rr^{2d}$ given by
Proposition~\ref{propcveq}, then $\sqrt{Q}$ is a norm on $\rr^{2d}$ so
that the space $({\mathcal P}_2, d_Q)$ is a complete metric space
(see~\cite{bolley-strasbourg} or~\cite[Chapter 6]{villani-stflour} for
instance).

Then Lemma~\ref{lemmacveq} below (see~\cite[Lemma
7.3]{carrillotoscani-portoercole} for instance) and the contraction
property of Proposition~\ref{propcveq} ensure the existence of a
unique stationary solution $\mu_{\infty}$ in ${\mathcal P}_2$
to~\eqref{eq:vfp}:

\begin{lemma}\label{lemmacveq}
 Let $(S, dist)$ be a complete metric space and $(T(t))_{t \geq 0}$
 be a continuous semigroup on $(S, dist)$ for which for all positive
 $t$ there exists $L(t) \in ]0,1[$ such that
$$
dist (T(t)(x), T(t)(y)) \leq L(t) \, dist (x,y)
$$
for all positive $t$ and $x,y$ in $S.$ Then there exists a unique
stationary point $x_{\infty}$ in $S$, that is, such that
$T(t)(x_{\infty}) = x_{\infty}$ for all positive $t.$
\end{lemma}

Moreover all solutions $(f_t)_{t \geq 0}$ with initial data $f_0$ in
${\mathcal P}_2$ converge to this stationary solution $\mu_{\infty},$
with
$$
d_Q(f_t, \mu_{\infty}) \leq  e^{-Ct} \, d_Q(f_0, \mu_{\infty}), \qquad t \geq 0.
$$

Finally, with $\bar{f_0} = \mu_{\infty}$, \eqref{pseudocontr} specifies into
$$
d(f_t, \mu_{\infty}) \leq C' e^{-C t} d(f_0,\mu_{\infty}), \qquad t
\geq 0
$$
which concludes the proof of Theorem~\ref{theocveq}. 
\end{proof}


\section{Particle approximation}\label{sect-prop}

The time-uniform propagation of chaos in Theorem \ref{th:chaos}
requires a time-uniform bound of the second moment of the solutions
to \eqref{eq:vfp}.

\begin{lemma}\label{2moment}
 Under Assumption $(\cA)$, for all positive $\alpha, \alpha'$ and
 $\beta$ there exists a positive constant $c$ such that $\ds \sup_{t
   \geq 0} \int_{\rr^{2d}} (\vert x \vert^2 + \vert v \vert^2) \,
 f_t(x,v) \, dx \, dv$ is finite for all solutions $(f_t)_{t \geq 0}$
 to~\eqref{eq:vfp} with initial datum $f_0$ in ${\mathcal P}_2$ and
 $\gamma, \delta$ in $[0, c).$
\end{lemma}

\begin{proof}[Proof of Lemma~\ref{2moment}]
Let $(f_t)_{t \geq 0}$ be a solution to~\eqref{eq:vfp} with initial
datum $f_0$ in ${\mathcal P}_2$, and let $a$ and $b$ be positive
numbers to be chosen later on. Then
\begin{multline*}
\frac{d}{dt} \int_{\rr^{2d}} \big( a \, \vert x \vert^2 + 2 \, x \cdot v + b \, \vert v \vert^2 \big) \, df_t(x,v) \\
= 2 b d + 2 \int_{\rr^{2d}} v \cdot (a x +  v) - (x + bv) \cdot \big(A(v) + B(x) + C *_x \rho[f_t](x) \big) \, df_t(x,v) 
\end{multline*}
where, by Young's inequality and assumption on $A,$ $B$ and $C,$
$$
- 2 \, x \cdot A(v) = -2 \, x \cdot \big(A(v) - A(0)\big) - 2 \, x \cdot A(0)  \leq 2 \, \alpha \, \vert x \vert  \, \vert v \vert - 2 x \cdot A(0), 
$$
$$
-2 \, x \cdot B(x) = -2 \,  x \cdot \big( \beta \, x + D(x) - D(0) + D(0) \big) \leq -(2 \beta - 2 \delta) \vert x \vert^2 - 2 \, x \cdot D(0),
$$
$$
-2 \, b\, v \cdot A(v) = - 2 \, b \, v \cdot (A(v) - A(0) + A(0) \big) \leq - 2 \, b \, \alpha' \vert v \vert^2 - 2 \, b \, v \cdot A(0),
$$
$$
-2 \, b \, v \cdot B(x) = -2 \, b \, v \cdot \big( \beta \, x + D(x) - D(0) + D(0) \big) \leq - 2 b \beta \, v \cdot x + 2 b \delta \,  \vert v \vert \, \vert x \vert - 2 b \, v \cdot D(0) ,
$$
\begin{eqnarray*}
- 2 \int_{\rr^{2d}} x \cdot C*_x \rho[f_t] (x) \, df_t(x,v)
&=& -  \int_{\rr^{4d}} (x-y) \cdot C(x-y) \, df_t(x,v) \, df_t(y,w) \\
& \leq & \gamma \int_{\rr^{4d}} \vert x-y \vert^2 \, df_t(x,v) \, df_t(y,w) \\
& \leq & 2 \, \gamma \int_{\rr^{2d}} \vert x \vert^2 \, df_t(x,v) 
\end{eqnarray*}
and
\begin{eqnarray*}
- 2 \, b \int_{\rr^{2d}} v \cdot C*_x \rho[f_t] (x) \, df_t(x,v) 
&=& - b \int_{\rr^{4d}} (v-w) \cdot C(x-y) \, df_t(x,v) \, df_t(y,w) \\
& \leq & b \, \gamma  \int_{\rr^{4d}} \vert v-w \vert\, \vert x-y \vert \,df_t(x,v) \,df_t(y,w)\\
& \leq & \frac{b \gamma}{2} \int_{\rr^{4d}} \big(\vert v- w \vert^2 +  \vert x-y \vert^2 \big) \, df_t(x,v) \, df_t(y,w)\\
& \leq & b \, \gamma \int_{\rr^{2d}} ( \vert x \vert^2 + \vert v \vert^2) df_t(x,v).
\end{eqnarray*}

Collecting all terms and using Young's inequality we obtain, with $a = \beta b$ and $\eta = \gamma + \delta$,
\begin{eqnarray*}
& &\frac{d}{dt} \int_{\rr^{2d}} (\beta \, b \, \vert x \vert^2 + 2 \, x \cdot v + b \, \vert v \vert^2) \, df_t(x,v) \\
& \leq & 
2 bd + (2 \alpha + 2 b \delta) \int \vert x  \vert \, \vert v \vert \, df_t(x,v) 
+\big[ b \gamma + 2 \gamma - 2 \beta +  2 \delta  \big] \int \vert x \vert^2 df_t(x,v)  \\
& & 
+ \big[ 2 + \gamma b - 2 \alpha'b \big] \int \vert v \vert^2 \, df_t(x,v) - 2 \big( A(0)+D(0) \big) \cdot \Big(\int x \, df_t(x,v) + b \int v \, df_t(x,v) \Big)\\
&  \leq &
2 bd + \Big( \frac{2}{\varepsilon} + \frac{b^2 \varepsilon}{2 \alpha^2} \Big) \vert A(0) + D(0) \vert^2  \\
& &
- \big[ 2 \beta - 2 \eta - \varepsilon - \eta b \big] \int \vert x \vert^2 \, df_t(x,v) - \Big[(2 \alpha' - \eta) b - 2 - \frac{4 \alpha^2}{\varepsilon} \Big] \int \vert v \vert^2 \, df_t(x,v)
\end{eqnarray*} 
for all positive $\varepsilon$.

Now, as in the proof of Proposition~\ref {propcveq}, with $\alpha$
replaced by $2 \alpha$, we get the existence of a positive constant
$\eta_0$, depending only on $\alpha, \alpha'$ and $\beta$, such that
for all $0 \leq \gamma + \delta < \eta_0$ there exist $b$ (and
$\varepsilon$) such that $Q(x,v) = \beta \, b \, \vert x \vert^2 + 2
\, x \cdot v + b \, \vert v \vert^2$ be a positive quadratic form on
$\rr^{2d}$ and such that
$$
\frac{d}{dt}  \int_{\rr^{2d}} Q(x,v) \, f_t(x,v) \, dx \, dv \leq C_1 - C_2 \int_{\rr^{2d}} (\vert x \vert^2 + \vert v \vert^2) \, f_t(x,v) \, dx \, dv \leq C_1 - C_3 \int_{\rr^{2d}} Q(x,v) \, f_t(x,v) \, dx \, dv
$$
for positive constants $C_i.$ It follows that
$$
\sup_{t \geq 0} \int_{\rr^{2d}} Q(x,v) \, f_t(x,v) \, dx \, dv < + \infty
$$
if initially $\ds \int_{\rr^{2d}} Q(x,v) \, f_0(x,v) \, dx \, dv < +
\infty$, that is,
$$
\sup_{t \geq 0} \int_{\rr^{2d}} ( \vert x \vert^2 + \vert v \vert^2) \, f_t(x,v) \, dx \, dv < + \infty
$$
if initially $f_0$ belongs to ${\mathcal P}_2.$ This concludes the
argument. 
\end{proof}

We now turn to the 

\begin{proof}[Proof of Theorem \ref{th:chaos}]
For each $1 \leq i \leq N$ the law $f_t$ of $(\bar{X}^i_t,
\bar{V}^i_t)$ is the solution to~\eqref{eq:vfp} with $f_0$ as initial
datum and the processes $(\bar{X}^{i}_t,\bar{V}^{i}_t)_{t \geq 0}$ and
${(X^{i,N}_t,V^{i,N}_t)}_{t\geq 0}$ are driven by the same Brownian
motion. In particular the differences $x_t^i = X^{i,N}_t -
\bar{X}^i_t$ and $v_t^i = V^{i,N}_t - \bar{V}^i_t$ evolve according to
$$
\left\{ \begin{array}{rcl}
             dx^i_t    & = &  v^i_t \, dt \\
            dv^i_t     & = &  - \big( A(V^{i,N}_t) - A(\bar{V}^i_t)+ \beta \, x^i_t + D(X^{i,N}_t) - D(\bar{X}^i_t) \big) \, dt 
            			 - \ds \frac{1}{N} \sum_{j=1}^N \big( C(X^{i,N}_t - X^{j,N}_t) - C *_x \rho[f_t](\bar{X}^i_t) \big) \, dt  
                         \end{array}
                 \right. 
$$
with $(x^i_0, v^i_0) = (0,0).$

Then, if $a$ and $b$ are positive constants to be chosen later on,
\begin{eqnarray*}
\frac{d}{dt} ( a \vert x^i_t \vert^2 + 2 \, x^i_t \cdot v^i_t + b \vert v^i_t \vert^2 ) 
& = &
2 \, a  \, x^i_t \cdot  v^i_t + 2 \vert v^i_t \vert^2  
- 2 \, x^i_t \cdot \big( A(V^{i,N}_t) - A(\bar{V}^i_t) + \beta x^i_t + D(X^{i,N}_t) - D(\bar{X}^i_t)  \big) \\
& & - 2 \, b \, v^i_t \cdot \big( A(V^{i,N}_t) - A(\bar{V}^i_t)  + \beta x^i_t  + D(X^{i,N}_t) - D(\bar{X}^i_t) \big)\\
& & -  \frac{2}{N} \sum_{j=1}^N  (x^i_t + b \, v^i_t) \cdot \big( C(X^{i,N}_t - X^{j,N}_t) - C *_x \rho[f_t](\bar{X}^i_t) \big).
\end{eqnarray*}

By the Young inequality and assumptions on $A$ and $D$, for all
positive $\varepsilon$ the third and fourth terms are bounded by above
according to
$$
-2 \, x^i_t \cdot (A(V^{i,N}_t) - A(\bar{V}^i_t)) \leq 2 \vert x^i _t \vert \,  \vert A(V^{i,N}_t) - A(\bar{V}^i_t) \vert \leq 2 \, \alpha \, \vert x^i_t \vert \, \vert v^i_t \vert \leq \frac{\varepsilon}{2} \vert x^i _t \vert^2 + \frac{2 \alpha^2}{\varepsilon} \vert v^i_t \vert^2,
$$
$$
- 2 \, x^i_t \cdot (D(X^{i,N}_t) - D(\bar{X}^i_t)) = - 2 b (X^i_t - \bar{X}^i_t ) \cdot (D(X^{i,N}_t) - D(\bar{X}^i_t)) \leq 2 \, \delta \, \vert x^i_t \vert^2,
$$
$$
- 2 \, b \,  v^i_t \cdot (A(V^{i,N}_t) - A(\bar{V}^i_t)) = - 2 \, b  (V^{i,N}_t - \bar{V}^i_t ) \cdot (A(V^{i,N}_t) - A(\bar{V}^i_t)) \leq - 2 \, b \, \alpha \vert v^i _t \vert^2 
$$
and
$$
- 2 \, b \, v^i_t \cdot (D(X^{i,N}_t) - D(\bar{X}^i_t)) \leq 2 \, b \, \delta \vert v^i _t \vert \,  \vert x^i_t \vert \leq b \, \delta ( \vert x^i _t \vert^2 + \vert v^i _t \vert^2).
$$

Hence, with $a = \beta \, b$,
\begin{eqnarray*}
\frac{d}{dt}( \beta \, b \vert x^i_t \vert^2 + 2 \, x^i_t \cdot v^i_t + b \vert v^i_t \vert^2 )
&\leq& 
\big( \frac{\varepsilon}{2} - 2 \beta + 2 \delta + \delta b \big) \vert x^i_t \vert^2 + \big( 2 + \frac{2 \alpha^2}{\varepsilon} - 2 \alpha' b + \delta b \big) \vert v^i_t \vert^2 \\
& & - \frac{2}{N} \sum_{j=1}^N  (x^i_t + b \, v^i_t) \cdot \big( C(X^{i,N}_t - X^{j,N}_t) - C *_x \rho[f_t](\bar{X}^i_t) \big).
\end{eqnarray*}

Moreover, by symmetry, $\ee \vert x^i_t \vert^2$, $\ee x^i_t \cdot
v^i_t$, ... are independent of $i=1, \dots, N$, so that, by averaging
on $i$,
\begin{eqnarray}
\frac{d}{dt} \ee \big[ \beta \, b \vert x^1_t \vert^2 + 2 \, x^1_t \cdot v^1_t + b \vert v^1_t \vert^2 \big] 
&\leq  &
- \big( 2 \beta - 2 \delta - \frac{\varepsilon}{2} - \delta b \big)  \ee \vert x^1_t \vert^2 - \big( (2 \alpha' - \delta) b - 2 - \frac{2 \alpha^2}{\varepsilon} \big) \ee \vert v^1_t \vert^2 \nonumber\\
& & - \frac{2}{N^2} \! \sum_{i,j=1}^N  \ee \Big[ (x^i_t + b \, v^i_t) \cdot \big( C(X^{i,N}_t  - X^{j,N}_t) - C *_x \rho[f_t](\bar{X}^i_t) \big) \Big]. \qquad \label{lastterm}
\end{eqnarray}

\medskip

We decompose the last term in~\eqref{lastterm} according to
$$
C(X^{i,N}_t - X^{j,N}_t) - C *_x \rho[f_t](\bar{X}^i_t)  = C(X^{i,N}_t - X^{j,N}_t) - C(\bar{X}^i_t - \bar{X}^j_t) + C(\bar{X}^i_t - \bar{X}^j_t) - C *_x \rho[f_t](\bar{X}^i_t)
$$
which leads to estimating four terms:

{\bf 1.} By symmetry and assumption on $C$,
\begin{eqnarray*}
- \! \sum_{i,j=1}^N \! \ee  \Big[ x^i_t \cdot \big( C(X^{i,N}_t - X^{j,N}_t) - C(\bar{X}^i_t - \bar{X}^j_t) \big) \Big]
\! \! \! \! & = &
- \sum_{i,j=1}^N \ee \, \Big[ (X^{i,N}_t - \bar{X}^i_t)   \cdot  \big( C(X^{i,N}_t - X^{j,N}_t) - C(\bar{X}^i_t - \bar{X}^j_t) \big) \Big] \\
& = &
\! \! - \frac{1}{2} \sum_{i,j=1}^N \ee \, \Big[  \big((X^{i,N}_t - X^{j,N}_t) - (\bar{X}^i_t - \bar{X}^j_t) \big) \\
& &  \qquad \qquad \qquad \cdot \;  \big( C(X^{i,N}_t - X^{j,N}_t) - C(\bar{X}^i_t - \bar{X}^j_t) \big) \Big] 
\\
& \leq & 
\frac{\gamma}{2} \sum_{i,j=1}^N \ee \,  \big\vert (X^{i,N}_t - X^{j,N}_t) - (\bar{X}^i_t - \bar{X}^j_t) \big\vert^2 \\
& = &
\frac{\gamma}{2} \sum_{i,j=1}^N \ee \,  \big\vert (X^{i,N}_t - \bar{X}^i_t) - (\bar{X}^{j,N}_t - \bar{X}^j_t) \big\vert^2 \\
& = &
\gamma \sum_{i,j=1}^N \ee \vert x^i_t \vert^2 - \gamma \ee \Big\vert \sum_{i=1}^N (X^{i,N}_t - \bar{X}^i_t) \Big\vert^2\\
& \leq &
\gamma N^2 \vert x^1_t \vert^2.
\end{eqnarray*}

{\bf 2.} By assumption on $C$ and the Young inequality,
\begin{eqnarray*}
- \sum_{i,j=1}^N \ee \, v^i_t \cdot \big( C(X^{i,N}_t - X^{j,N}_t) - C(\bar{X}^i_t - \bar{X}^j_t) \big)
& = &
- \frac{1}{2} \sum_{i,j=1}^N \ee \big[ (v^i_t - v^j_t)  \cdot \big( C(X^{i,N}_t - X^{j,N}_t) - C(\bar{X}^i_t - \bar{X}^j_t) \big) \big]\\
& \leq & 
\frac{\gamma}{2} \sum_{i,j=1}^N \ee \big[ \vert v^i _t - v^j_t \vert \,   \big\vert (X^{i,N}_t - X^{j,N}_t) - (\bar{X}^i_t - \bar{X}^j_t) \big\vert  \big] \\
& \leq &
\frac{\gamma}{2} \sum_{i,j=1}^N \ee \, \Big[ \frac{1}{2} \vert v^i_t - v^j_t \vert^2 + \frac{1}{2}  \big\vert (X^{i,N}_t - \bar{X}^i_t) - (X^{j,N}_t - \bar{X}^j_t) \big\vert^2 \Big] \\
& \leq &
\frac{\gamma}{2} N^2  \, \ee ( \vert v^1_t \vert^2 +  \vert x^1_t \vert^2).
\end{eqnarray*}

{\bf 3.} For each $i=1, \dots, N$, and again by the Young inequality
$$
- 2 \, \ee \, \Big[ x^i_t \cdot \sum_{j=1}^N \Big( C(\bar{X}^i_t - \bar{X}^j_t) - C *_x \rho[f_t](\bar{X}^i_t) \Big) \Big]  \leq L N \ee \, \vert x^i_t \vert^2 + \frac{1}{L N} \ee  \Big\vert \sum_{j=1}^N \Big( C(\bar{X}^i_t - \bar{X}^j_t) - C *_x \rho[f_t](\bar{X}^i_t) \Big) \Big\vert^2
$$
for any positive constant $L,$ where the last expectation is
$$
\sum_{j=1}^N \ee \big\vert C(\bar{X}^i_t - \bar{X}^j_t) - C *_x \rho[f_t](\bar{X}^i_t) \big\vert^2 + \sum_{j \neq k} \ee \Big[ \big(C(\bar{X}^i_t - \bar{X}^j_t) - C *_x \rho[f_t](\bar{X}^i_t) \big) \cdot \big( C(\bar{X}^i_t - \bar{X}^k_t) - C *_x \rho[f_t](\bar{X}^i_t) \big) \Big].
$$

First of all, $C$ is odd, so $C(0)=0$ and hence $ \vert C(z) \vert
\leq \gamma \, \vert z \vert.$ Then, for each $j= 1, \dots, N,$
\begin{eqnarray*}
\ee \big\vert C(\bar{X}^i_t - \bar{X}^j_t) - C *_x \rho[f_t](\bar{X}^i_t) \big\vert^2
& \leq &
2 \, \ee \big\vert C(\bar{X}^i_t - \bar{X}^j_t) \big\vert^2 + 2 \, \ee \big\vert C *_x \rho[f_t](\bar{X}^i_t) \big\vert^2\\
& \leq & 2 \, \gamma^2 \Big[ \ee \big\vert \bar{X}^i_t - \bar{X}^j_t \big\vert^2 + \ds \int_{\rr^{4d}} \vert y - x \vert^2 f_t(x,v) \, f_t(y,w) \, dx \, dv \, dy \, dw \Big]\\
& \leq & 8 \, \gamma^2 \, \int_{\rr^{2d}} \vert x \vert^2 \, f_t(x,v) \, dx \, dv\\
& \leq & M
\end{eqnarray*}
for a constant $M$, provided $\gamma$ and $\delta$ are small enough for the conclusion of Lemma \ref{2moment} to
hold. The constant $M$ depends on the initial moment $\ds  \int_{\rr^{2d}}  \big( \vert x \vert^2 + \vert v \vert^2 \big) \, f_0(x,v) \, dx \, dv$ and the coefficients of the equation, but not on $t$ or $N.$

Then, for all $j \neq k$, 
\begin{multline*}
\ee \Big[ \big(C(\bar{X}^i_t - \bar{X}^j_t) - C *_x \rho[f_t](\bar{X}^i_t) \big) \cdot \big( C(\bar{X}^i_t - \bar{X}^k_t) - C *_x \rho[f_t](\bar{X}^i_t) \big) \Big] \\
\begin{array}{rl}
&= \ee_{\bar{X}^i_t} \Big[ \Big (\ee_{\bar{X}^j_t} \big[C(\bar{X}^i_t - \bar{X}^j_t) - C *_x \rho[f_t](\bar{X}^i_t) \big] \Big) \cdot \Big( \ee_{\bar{X}^k_t} \big[C(\bar{X}^i_t - \bar{X}^k_t) - C *_x \rho[f_t](\bar{X}^i_t) \big] \Big) \Big] \\
&= \ee_{\bar{X}^i_t} \big[ 0 \big] =0
\end{array}
\end{multline*}
since $\bar{X}_t^j$ and $\bar{X}_t^k$ are independent and have law $\rho[f_t]$.

To sum up,
$$
- 2 \sum_{i,j=1}^N  \ee \, \Big[ x^i_t \cdot \big( C(\bar{X}^i_t - \bar{X}^j_t) - C *_x \rho[f_t](\bar{X}^i_t) \big) \Big] \leq L \, N^2 \, \ee \vert x_t^1 \vert^2 + \frac{M}{L} N.
$$

{\bf 4.} In the same way for any positive $L'$ we obtain the bound
$$
- 2 \sum_{i,j=1}^N  \ee \, \Big[ v^i_t \cdot \big( C(\bar{X}^i_t - \bar{X}^j_t) - C *_x \rho[f_t](\bar{X}^i_t) \big) \Big] \leq L' \, N^2 \, \ee \vert v_t^1 \vert^2 + \frac{M}{L'} N.
$$

\medskip

Collecting all terms and letting for instance $\ds L =
\frac{\varepsilon}{2}$ and $\ds L' = \frac{2 \alpha^2}{b \varepsilon}$,
it follows from~\eqref{lastterm} that there exists a positive constant
$c$ such that for all $\gamma, \delta$ in $[0, c)$ there exists a
constant $M$ such that
\begin{multline*}
\frac{d}{dt} \ee \big[ \beta \, b \vert x^1_t \vert^2 + 2 \, x^1_t \cdot v^1_t + b \vert v^1_t \vert^2 \big] \\
\leq
- (2 \beta - \varepsilon - 2 \eta  - \eta b) \ee \vert x^1_t \vert^2 - \big( (2 \alpha' - \eta) b  - 2 - \frac{4 \alpha^2}{\varepsilon}  \big) \ee \vert v^1_t \vert^2  +  \frac{M}{N} \Big( \frac{2}{\varepsilon} + \frac{\varepsilon b^2}{2 \alpha^2} \Big) \cdot
\end{multline*}
for all positive $t, b$ and $\varepsilon,$ where $\eta = \gamma + \delta$.

Now, as in the proof of Proposition~\ref {propcveq}, with $\alpha$
replaced by $2 \alpha$, we get the existence of a positive constant
$\eta_0$, depending only on $\alpha, \alpha'$ and $\beta$, such that
for all $0 \leq \gamma + \delta < \eta_0$ there exist $b$ (and
$\varepsilon$) such that $Q(x,v) = \beta \, b \, \vert x \vert^2 + 2
\, x \cdot v + b \, \vert v \vert^2$ be a positive quadratic form on
$\rr^{2d}$ and such that
$$
\frac{d}{dt} \ee \, Q(x_t^1, v_t^1) \leq - C_1 \, \ee \big[ \vert x_t^1 \vert^2 + \vert v_t^1 \vert^2 \big] + \frac{C_2}{N}
$$
for all $t \geq 0$ and for positive constants $C_1$ and $C_2$, also
depending on $f_0$ through its second moment, but not on $N.$ In turn
this is bounded by $- C_3 \, \ee \, Q(x_t^1, v_t^1) + \ds
\frac{C_2}{N}$, so that
$$
\sup_{t \geq 0} \, \ee \, Q(x_t^1, v_t^1) \leq \frac{C_4}{N}
$$
and finally
$$
\sup_{t \geq 0}  \ee \, \big[ \vert X_t^{1,N} - \bar{X}_t^1 \vert^2 + \vert V_t^{1,N} - \bar{V} _t^1 \vert^2 \big] \leq \frac{C}{N}
$$
where the constant $C$ depends on the parameters of the equation and
on the second moment of $f_0$, but not on $N.$ This concludes the
proof of Theorem~\ref{th:chaos}. 
\end{proof}

\begin{rem}
One can prove a contraction property for the particle system, similar to Proposition~\ref{propcveq} for the Vlasov-Fokker-Planck equation: if $f_0$ is an initial datum in ${\mathcal P}_2$  we let $f_t^{(1,N)}$ be the common law at time $t$ of any of the $N$ particles $(X^{i,N}_t, V^{i,N}_t).$ Then there exists a positive constant $c$ such that, if $0 \leq  \gamma, \delta < c$, then there exist a positive constant $C$ and a positive quadratic form $Q$ on $\rr^{2d}$ such that
$$
d_Q(f_t^{(1,N)}, {\tilde f}_t^{(1,N)}) \leq e^{-C t} \, d_Q(f_0^{(1,N)}, {\tilde f}_0^{(1,N)}) = e^{-C t} \, d_Q(f_0, {\tilde f}_0)
$$
for all $t$ and all initial data $f_0$ and ${\tilde f}_0$ in ${\mathcal P}_2$. Here the form $Q$ and the constants $c$ and $C$ depend only on the coefficients of the equations, and not on $N.$ 
From this and Remark~\ref{rem:cvmarg}, and following~\cite{CGM}, one can recover the contraction property of Proposition~\ref{propcveq}, whence Theorem~\ref{theocveq}.
\end{rem}



\section{Transportation inequality and deviation result}\label{sect-TI}


This final section is devoted to the proof of Theorem~\ref{th:T2}. It is based on the idea, borrowed to~\cite{DGW}, of proving a $T_2$ transportation inequality not only for the law $f_T^{(N)}$ at time $T$, but for the whole trajectory up to time $T;$ this transportation inequality will be proved by means of stochastic calculus, a coupling argument, a clever formulation of the relative entropy of two trajectory laws and a change of metric as in the previous sections; it will imply the announced transportation inequality by projection at time~$T.$

We only sketch the proof, emphasizing the main steps and refering to the previous sections and to~\cite{DGW} for further details.

\medskip

We equip the space ${\mathcal C}$ of $\rr^{2dN}$-valued  continuous functions on $[0,T]$ with the $L^2$ norm
and consider the space ${\mathcal P} (\mathcal C)$ of Borel probability measures on $\mathcal C$, equipped with the Wasserstein distance defined by the cost $\Vert \gamma_1 - \gamma_2 \Vert_{L^2}^2$ for $ \gamma_1,  \gamma_2 \in \mathcal C.$


We write Equation~\eqref{eq:sdesys} on the particle system $(X_t^{(N)},V^{(N)}_t)_{t \geq 0}$ in the form
$$d(X_t^{(N)},V^{(N)}_t)=\sigma^{(N)}(X_t^{(N)},V^{(N)}_t)dW^{(N)}_t+b^{(N)}(X_t^{(N)},V^{(N)}_t)dt$$
for some coefficients $\sigma^{(N)}$ and $b^{(N)}.$  

\smallskip

Let $\dP \in {\mathcal P} (\mathcal C)$ be the law of the trajectory $(X^{(N)},V^{(N)}) = (X_t^{(N)},V^{(N)}_t)_{0 \leq t \leq T}$ of the particles, all of them starting from the deterministic point $(x_0, v_0) \in \rr^{2d}.$ 

The transportation inequality for $\dP$, which it is sufficient to prove for laws $\dQ$ absolutely continuous with respect to $\dP,$ will obtained in two steps.\\

{\it Step 1.} Following~\cite[Proof of Theorem 5.6]{DGW}, for every trajectory law $\dQ \in {\mathcal P} (\mathcal C)$, there exists $(\beta_t)_{t\in[0,T]}\in L^2([0,T],R^{2dN})$ such that $H(\dQ,\dP)=\frac12\dE^\dQ\int_0^T|\beta_t|^2dt$; moreover
$$
d(X_t^{(N)},V^{(N)}_t)=\sigma^{(N)}(X_t^{(N)},V^{(N)}_t)d\tilde W_t^{(N)}+b^{(N)}(X_t^{(N)},V^{(N)}_t)dt+\sigma^{(N)}(X_t^{(N)},V^{(N)}_t)\beta_tdt
$$
under the law $\dQ$,  where $\tilde W_t^{(N)}=W_t^{(N)}-\ds \int_0^t\beta_sds$ is a Brownian motion under $\dQ$. We now build a coupling between $\dQ$ and $\dP$ by letting $(\tilde X^{(N)},\tilde V^{(N)}) = (\tilde X_t^{(N)},\tilde V^{(N)}_t)_{0 \leq t \leq T}$ be the solution (under $\dQ$) of 
$$
d(\tilde X_t^{(N)},\tilde V^{(N)}_t)=\sigma^{(N)}(\tilde X_t^{(N)},\tilde V^{(N)}_t)d\tilde W_t^{(N)}+b^{(N)}(\tilde X_t^{(N)},\tilde V^{(N)}_t)dt,
$$
whose law under $\dQ$ is exactly $\dP$.\\

{\it Step 2.} In order to prove the $T_2$ inequality, and as in the previous sections, we change the metric induced on $\mathcal C$ by the $L^2$  norm and consider an equivalent positive quadratic form $Q(x,v) = a \vert x \vert^2 + 2 \, x \cdot v + b \vert v \vert^2.$ We control the quantity
$$
\dE^\dQ Q\big( (X_t^{(N)},V^{(N)}_t)-(\tilde X_t^{(N)},\tilde V^{(N)}_t) \big)
$$
by proving the existence of a positive constant $D$, independent of $N$ and $T$, such that
$$
\dE^\dQ \big[a|x_t|^2+2 x_t \cdot v_t+b|v_t|^2 \big]\le \! -D \!  \int_0^t\dE^\dQ \big[ |x_s|^2+|v_s|^2 \big] ds \, +\int_0^t\dE^\dQ \big[ \nabla Q(x_s,v_s) \cdot \sigma^{(N)}(X_s^{(N)},V^{(N)}_s) \beta_s \big] ds
$$
in the notation $x_t=X_t^{(N)}-\tilde X_t^{(N)}$ and $v_t=V^{(N)}_t-\tilde V^{(N)}_t.$
Then we bound the last term  by
$$
\varepsilon\int_0^t\dE^\dQ \big[|x_s|^2+|v_s|^2 \big] ds+{1\over \varepsilon}\int_0^t\dE^\dQ|\beta_s|^2ds
$$
and the transportation inequality for the trajectory law $\dP$ follows again by Gronwall's lemma, with a new constant $D$ independent of $T.$

The transportation inequality for the law $f_T^{(N)}$ at time $T$ finally follows by projection at time $T.$

\bigskip

We now turn to the deviation inequality in Theorem~\ref{th:T2}. First of all, if $h$ is a $1$-Lipschitz function on $\rr^{2d},$ then
\begin{multline*}
\ds \frac1N\sum_{i=1}^Nh(X^{i,N}_T,V^{i,N}_T)-   \int_{\rr^{2dN}} \frac1N\sum_{i=1}^N h(x_i,v_i) \; df_T^{(N)} (x_1, \dots, v_N)\\
\begin{array}{rl}
&\ds = \frac1N\sum_{i=1}^Nh(X^{i,N}_T,V^{i,N}_T)- \int h d\mu_\infty+ \int h \, d\mu_\infty- \int h \, df_T+ \int h \, df_T- \int h \, df_T^{(1,N)} \\
& \ds  \ge \frac1N\sum_{i=1}^Nh(X^{i,N}_T,V^{i,N}_T)- \int h \, d\mu_\infty  - d(f_T,\mu_\infty) - d(f_T,f^{(1,N)}_T)
\end{array}
\end{multline*}
by exchangeability. But, by Theorem~ \ref{theocveq} with $f_0 = \delta_{(x_0, v_0)},$ there exist two constants $C$ and $C'$, depending only on the coefficients of the equation, such that
$$
d(f_T,\mu_\infty) \leq C' e^{-C T} \, d(f_0,\mu_\infty)
$$
where $f_T$ is the solution at time $T$ of Equation~\eqref{eq:vfp} with initial datum $f_0 = \delta_{(x_0, v_0)}.$ Moreover, by Remark~\ref{rem:cvmarg}, there exists a constant $C''$, depending only on the equation and on $(x_0, v_0)$, such that
$$
d(f_T,f^{(1,N)}_T) \leq \frac{C''}{\sqrt{N}}.
$$
Hence
\begin{multline*}
\dP \left[\frac1N\sum_{i=1}^Nh(X^{i,N}_T,V^{i,N}_T)- \int_{\rr^{2d}}  h\, d\mu_\infty \ge r + D'\left(\frac{1}{\sqrt{N}}+e^{-CT}\right)\right]
\\
\le
\dP \left[ \frac1N\sum_{i=1}^Nh(X^{i,N}_T,V^{i,N}_T) - \int_{\rr^{2dN}} \frac1N\sum_{i=1}^N h(x_i,v_i) \; df_T^{(N)} (x_1, \dots, v_N) \geq r \right]
\end{multline*}
where $D' = \max( C' d(f_0,\mu_\infty), C'')$ depends on $(x_0, v_0).$

Now the law $f_T^{(N)}$ satisfies a $T_2$ inequality on $\rr^{2dN}$ with constant $D$, hence a Gaussian deviation inequality for Lipschitz functions (see \cite{bobkov-gotze}); moreover the map $(x_1, \dots, v_N) \mapsto \ds  \frac1N\sum_{i=1}^N h(x_i,v_i)$ is $\ds \frac1{\sqrt{N}}$-Lipschitz on $\rr^{2dN}$, so the probability on the right-hand side is bounded by
$$
\exp\left(-\frac{Nr^2}{2D}\right)
$$
for all $r \geq 0.$
This concludes the proof of Theorem~\ref{th:T2}.

\footnotesize
\nocite{*}
\bibliography{bgm}

\bibliographystyle{amsplain}

\bigskip

{\footnotesize %
 \noindent Fran\c cois \textsc{Bolley}, %
 \noindent\url{mailto:bolley(AT)ceremade.dauphine(DOT)fr}

 \medskip
 \noindent\textsc{Ceremade, UMR CNRS 7534 \\
   Universit\'e Paris-Dauphine, Place du Mar\'echal De Lattre De
   Tassigny, F-75775 Paris \textsc{Cedex} 16}

\bigskip

\noindent Arnaud \textsc{Guillin}, corresponding author, %
\noindent\url{mailto:guillin(AT)math.univ-bpclermont(DOT)fr}

\medskip

\noindent\textsc{UMR CNRS 6620, Laboratoire de Math\'ematiques\\
  Universite Blaise Pascal, avenue des Landais,  F-63177 Aubiere
  \textsc{Cedex}}

\bigskip

 \noindent Florent \textsc{Malrieu},  %
 \url{mailto:florent.malrieu(AT)univ-rennes1(DOT)fr}

 \medskip

 \noindent\textsc{UMR CNRS 6625, Institut de Recherche Math\'ematique
   de Rennes (IRMAR) ; \\ Universit\'e de Rennes I, Campus de Beaulieu,
   F-35042 Rennes \textsc{Cedex}}

\vfill

\begin{flushright}\texttt{Compiled \today.}\end{flushright}

}

\end{document}